\documentclass[reqno,11pt]{amsart}

\usepackage{euscript}
\usepackage{enumerate}
\usepackage{amsmath}
\usepackage{mathrsfs}
\usepackage{mathtools}
\usepackage{verbatim}
\usepackage[initials]{amsrefs}
\usepackage{amssymb}
\usepackage{comment}
\usepackage{color}
\usepackage{hyperref}
\usepackage{relsize}

\theoremstyle{plain}
\newtheorem{thm}{Theorem}[section]
\newtheorem{prop}[thm]{Proposition}

\newtheorem{lemma}[thm]{Lemma}

\theoremstyle{remark}
\newtheorem{rem}[thm]{Remark}

\theoremstyle{definition}
\newtheorem{defi}[thm]{Definition}

\newcount\theTime
\newcount\theHour
\newcount\theMinute
\newcount\theMinuteTens
\newcount\theScratch
\theTime=\number\time
\theHour=\theTime
\divide\theHour by 60
\theScratch=\theHour
\multiply\theScratch by 60
\theMinute=\theTime
\advance\theMinute by -\theScratch
\theMinuteTens=\theMinute
\divide\theMinuteTens by 10
\theScratch=\theMinuteTens
\multiply\theScratch by 10
\advance\theMinute by -\theScratch

\def\today{{\number\day\space
 \ifcase\month\or
  January\or February\or March\or April\or May\or June\or
  July\or August\or September\or October\or November\or December\fi
 \space\number\year}}

\newcommand\smd[2]{\underset{#2}{#1}}
\newcommand\smdp[3]{\overset{#3}{\smd{#1}{#2}}}
\newcommand\Cpx{{\mathbb{C}}}
\newcommand\Nats{{\mathbb{N}}}
\newcommand\Fc{{\mathcal{F}}}
\newcommand\Mcal{{\mathcal{M}}}
\newcommand\Nc{{\mathcal{N}}}
\newcommand\Nct{{\widetilde\Nc}}
\newcommand\Pc{{\mathcal{P}}}
\newcommand\Qc{{\mathcal{Q}}}
\newcommand\Bt{{\widetilde B}}
\newcommand\gammat{{\tilde\gamma}}

\newcommand\id{{\operatorname{id}}}
\newcommand\restrict{{\upharpoonright}}
\newcommand\lspan{\mathrm{span}\,}

\newcommand\bigfreeprod{\mathlarger{\mathlarger{\mathlarger{*}}}}
\newcommand\rankns{\text{r}_{\text{ns}}}

\begin{document}

\title[Free products]{Free products and rescalings involving non-separable abelian von Neumann algebras}

\author[Dykema]{Ken Dykema}
 \address{Ken Dykema, Department of Mathematics, Texas A\&M University, College Station, TX 77843-3368, USA.}
 \email{kdykema@math.tamu.edu}
 
 \author[Zhao]{Junchen Zhao}
 \address{Junchen Zhao, Department of Mathematics, Texas A\&M University, College Station, TX 77843-3368, USA.}
 \email{junchenzhao@tamu.edu}

\begin{abstract} 
For a self-symmetric tracial von Neumann algebra $A$, we study rescalings of $A^{*n} * L\mathbb{F}_r$ for $n \in \mathbb{N}$ and $r \in (1, \infty]$ and use them to obtain an interpolation $\mathcal{F}_{s,r}(A)$ for all real numbers $s>0$ and $1-s < r \leq \infty$. We get formulas for their free products, and free products with finite-dimensional or hyperfinite von Neumann algebras. In particular, for any such $A$, we can compute compressions $(A^{*n})^t$ for $0<t<1$, and the Murray-von Neumann fundamental group of $A^{*\infty}$. When $A$ is also non-separable and abelian, this answers two questions in Section 4.3 of \cite{NON_ISOM}.

\end{abstract}

\date{\today}

\maketitle

\section{Introduction}

The free group factors $L\mathbb{F}_n$ for natural numbers $n \geq 2$ are among the first examples of II$_1$ factors considered by Murray and von Neumann in the 1940s in \cite{MvN43}. They showed that $L\mathbb{F}_2 \not \cong \mathcal{R}$, where $\mathcal{R}$ is the hyperfinite II$_1$ factor. However, the free group factor isomorphism problem, first asked in 1967 by Kadison, namely,
 whether $L\mathbb{F}_n \cong L\mathbb{F}_m$ or not for $n \neq m$, remains open until this day. In the 1990s, using Voiculescu's theory of free probability, the first author in \cite{IFGF} and R\u adulescu in \cite{Rad94} independently found a continuous family of II$_1$ factors $L\mathbb{F}_r$ for $r \in (1, \infty]$ as an interpolation of the free group factors. The family $L\mathbb{F}_r$ satisfies the following rescaling and free product addition formulas:
\[
	(L\mathbb{F}_r)^t = L\mathbb{F}_{1 + \frac{r-1}{t^2}}
\]
and
\[
	L\mathbb{F}_r * L\mathbb{F}_{r'} = L\mathbb{F}_{r + r'}
\]
for all $1 < r, r' \leq \infty$ and $0 < t < \infty$. These formulas led to the proof of the dichotomy that the interpolated free group factors, and thus the free group factors, are either all non-isomorphic or all isomorphic to each other.

 Recently in Theorem 1.1 of \cite{NON_ISOM}, Boutonnet, Drimbe, Ioana, and Popa settled the non-separable version of the free group factor isomorphism problem: for any non-separable diffuse abelian tracial von Neumann algebra $A$, the II$_1$ factors $A^{*n}$ are all non-isomorphic and have trivial fundamental group, for all natural numbers $n \geq 2$. Examples of such abelian algebras $A$ include $L\mathbb{R}$, namely, the group von Neumann algebra of $\mathbb{R}$ taken with discrete topology,
and the ultrapower $(L\mathbb{Z})^\omega$, where $\omega$ is a free ultrafilter on $\mathbb{N}$. In Proposition 4.1 and Corollary 4.4 in \cite{NON_ISOM}, using tools in free probability, compressions $(A^{*n})^{t}$ are computed to be
\begin{equation} \label{eq:A_*n^t}
	(A^{*n})^t = A^{*(n/t)} * L\mathbb{F}_{\frac{n-1}{t^2} - \frac{n}{t} + 1}
\end{equation}
 for any $n = 2,3,\ldots, \infty$, $t = 1/m$ for some $m \in \mathbb{N}$, and any \textit{homogeneous} abelian tracial von Neumann algebra $A$, namely, such that for every $k \in \mathbb{N}$, there exists a partition of unity into $k$ projections $p_1, \ldots, p_k \in A$ such that for every $1 \leq j \leq k$ we have that $\tau(p_j) = 1/k$ and $(Ap_j, k \tau |_{Ap_j}) \cong (A, \tau)$ (see Section 4.3 of \cite{NON_ISOM}). Our goal is to study rescalings of $(A^{*n})^t$ for more general $t \in (0, \infty)$.

In order to avoid tracking different pieces of $A$ in the rescalings of the free product, we impose a condition that roughly says $A$ can be continuously decomposed into pieces that are all the same as $A$ itself in a trace-scaling way:

\begin{defi} \label{def:self_symmetric}
	A tracial von Neumann algebra $(A, \tau)$ is called \textit{self-symmetric} if, for all $0 < t < 1$, there exists a central projection $p \in A$ with $\tau(p) = t$ and tracial isomorphisms $(Ap, t^{-1} \tau |_{Ap}) \cong (A, \tau) \cong (A(1-p), (1-t)^{-1} \tau |_{A(1-p)})$. 
\end{defi}

\begin{rem}
	Note here we are not requiring $A$ to be abelian itself. Rather, we just require $A$ to have central projections that do the job. Although most, if not all, literature related to our results in this paper concerns the cases where $A$ is abelian, we decide to use this definition that witnesses the interpolation of free product phenomenon for a more general class of tracial von Neumann algebras. Note that direct sums of self-symmetric tracial von Neumann algebras are still self-symmetric. Abelian examples include $L\mathbb{Z}$, $(L\mathbb{Z})^\omega$, and $(L\mathbb{Z})^\omega \oplus (L\mathbb{Z})$. Non-abelian examples include any self-symmetric abelian tracial von Neumann algebra tensor product with any non-abelian tracial von Neumann algebra. 
\end{rem}

\begin{rem}
	Note that for an abelian tracial von Neumann algebra $(A, \tau)$, it follows trivially from definitions that being self-symmetric implies being homogeneous, and being homogeneous implies being diffuse. We actually show in Theorem \ref{thm:abvNa} that any diffuse abelian tracial von Neumann algebra is self-symmetric, regardless of separability.  
\end{rem}
 
For a self-symmetric tracial von Neumann algebra $A$, $s \in (0, \infty)$, and $r \in (1-s, \infty]$, we define II$_1$-factors $\mathcal{F}_{s,r}$ (Definition \ref{def:IFGF}) so that $\mathcal{F}_{s,r} $ is isomorphic to $ A^{*s} * L\mathbb{F}_r$ when $s \in \mathbb{N}$ and $r \in (1, \infty]$. We prove the rescaling formula
 \begin{equation}\label{eq:F_rescale}
 	(\mathcal{F}_{s,r})^t \cong \mathcal{F}_{\frac{s}{t},\frac{s+r-1}{t^2} - \frac{s}{t} + 1}  
 \end{equation}
 for all $s \in (0, \infty)$, $r \in (1-s, \infty]$, $t \in (0, \infty)$. We also prove the addition formula for free products
 \begin{equation}\label{eq:F_addition}
 	\mathcal{F}_{s, r} * \mathcal{F}_{s', r'} \cong \mathcal{F}_{s + s', r + r'}
 \end{equation}
 for all $s, s' \in (0, \infty)$, $r \in (1-s, \infty]$, $r' \in (1-s', \infty]$. 

 We comment that, by results in \cite{NON_ISOM}, if $A$ is also non-separable and abelian, then $\mathcal F_{s,r} \not \cong \mathcal F_{s',r'} $ whenever $s \neq s'$. We will also develop theorems about countably infinite free products of members of the family $\mathcal F_{s,r}$. Combining with our formulas (\ref{eq:F_rescale}) and (\ref{eq:F_addition}), those allow us to show the fundamental group of $A^{*\infty}$ is $\mathbb R_+^*$, for all self-symmetric tracial von Neumann algebras $A$. 
 
\section{Background and Notation}

\subsection{Tracial von Neumann algebras and their isomorphisms}

A tracial von Neumann algebra is a pair $(M,\tau)$ of a von Neumann algebra $M$ endowed with a normal, faithful, tracial state $\tau$. By an isomorphism between tracial von Neumann algebras $(M,\tau_M)$ and $(N, \tau_N)$, we mean there exists a trace-preserving $*$-isomorphism, namely, $\phi : M \rightarrow N$ such that $\tau_N \circ \phi = \tau_M$ on $M$. The isomorphisms throughout this paper will all be tracial. We, however, simply write as $M \cong N$ when there is no ambiguity about the traces involved. 

\subsection{Rescaling of a II\texorpdfstring{$_1$}{1} factor}


For a II$_1$ factor $M$ and $t \in (0, \infty)$, we denote the rescaling of $M$ by $t$ as $M^t$. These were first considered by Murray and von Neumann in \cite{MvN43} and they used them to define the fundamental group of a II$_1$ factor.

\subsection{Tracial ultrapowers and non-separability} See Section 2.2 of \cite{ultrapower} for an excellent exposition of ultrafilters and ultrapowers. We observe the following corollary to the proofs of Lemma 2.2 and 2.3(1) in \cite{ultrapower}: for any free ultrafilter $\omega$ on $\mathbb N$, $(L\mathbb Z)^\omega$ is not only non-separable but also \textit{purely non-separable}, which in the abelian case means having no separable direct summand.

\subsection{Free dimension}  Free dimension, first defined in~ \cite{FDIM}, is a tool for computing various free products involving finite-dimensional or hyperfinite von Neumann algebras, interpolated free group factors, and direct sums of those.  See~\cite{FDIM} for details.

\subsection{Weighted direct sum} For a finite or countably infinite index set $I$, and tracial von Neumann algebras $\{(A_i, \tau_i)\}_{i \in I}$, we denote as 
\[
	A = \bigoplus_{i \in I} \overset{p_i}{\underset{t_i}{A_i}}
\]
the von Neumann algebra direct sum of $\{A_i\}_{i \in I}$ equipped with the trace $\tau((a_i)_{i \in I}) = \sum_{i \in I} t_i\tau_i(a_i)$, where $0 < t_i < 1$ for all $i \in I$, $\sum_{i \in I} t_i = 1$, and $p_i$ is the projection in $A$ that is the identity in $A_i$ and zero elsewhere. 

Here we prove a lemma on weighted direct sums of self-symmetric tracial von Neumann algebras that will be of use later.

\begin{lemma} \label{lem:direct_sum_self_symm}
	Let $I$ be a finite or countably infinite index set, and $\{t_i\}_{i \in I}$ be numbers in $(0, 1)$ such that $\sum_{i \in I} t_i = 1$. Let $(A,\tau)$ be a self-symmetric tracial von Neumann algebra. Then 
	\[
		A \cong \bigoplus_{i \in I} \underset{t_i}{A_i}
	\]
	where each $A_i$ is a copy of $A$. 
\end{lemma}

\begin{proof}
	Using the pair of isomorphisms in Definition \ref{def:self_symmetric}, we can recursively choose mutually orthogonal central projections $\{p_i\}$ in $A$ with $\tau(p_i) = t_i$ and $(Ap_i, t_i^{-1}\tau|_{Ap_i}) \cong (A, \tau)$. Then we have $\sum_{i \in I} p_i = 1_A$ since
	\[
		\tau(\sum_{i \in I} p_i) = \sum_{i \in I} \tau(p_i) = \sum_{i \in I} t_i = 1,
	\]
	and thus
	\[
		A \cong \bigoplus_{i \in I} Ap_i \cong \bigoplus_{i \in I} \underset{t_i}{A_i}.
	\]
\end{proof}

\subsection{Diffuse abelian tracial von Neumann algebras}

See~\cite{abvNa} for a comprehensive account of abelian von Neumann algebras. It contains classification results of \textit{von Neumann measure algebras} (von Neumann algebras with a faithful normal semi-finite weight) that enable us to prove that every diffuse abelian tracial von Neumann algebra is self-symmetric (see Theorem \ref{thm:abvNa} below). 

\begin{defi}[Section 9.1 in \cite{abvNa}] \label{def:a_kappa}
	Let $\kappa$ be a cardinal. Let $(\mathcal A_\kappa, \tau_\kappa)$ denote the tensor product of the family of von Neumann probability algebras $\{(A_i,\tau_i)\}_{i \in \kappa}$, where each $(A_i,\tau_i)$ is the tracial von Neumann algebra
	\[
		\underset{1/2}{\mathbb C} \oplus \underset{1/2}{\mathbb C}.
	\]
	In particular, $\tau_\kappa$ is a normal faithful tracial state on $\mathcal A_\kappa$.
\end{defi}

\begin{defi}[Definition 9.7 in \cite{abvNa}]
	Let $A$ be an abelian von Neumann algebra. We denote by $\mathfrak{k}(A)$ the smallest cardinality of a set $P$ of projections in $A$ not containing $1_A$ such that $P \cup \{1_A\}$ generates $A$ as a von Neumann algebra. 
\end{defi}

\begin{defi}[Definition 9.9 in \cite{abvNa}]
	Let $\kappa$ be a cardinal. We say that an abelian von Neumann algebra $A$ is \textit{$\kappa$-homogeneous} if for any non-zero projection $p \in A$, we have $\kappa = \mathfrak{k}(A) = \mathfrak{k}(Ap)$. 

\end{defi}

The following is a special case of Theorem 9.16 in \cite{abvNa}.
\begin{thm}[] \label{thm:k-homo_implies_isomo}
Any two 
$\kappa$-homogeneous abelian von Neumann measure algebras equipped with normal, faithful, tracial states
are isomorphic, via a trace-preserving von Neumann algebra isomorphism.
\end{thm}

\begin{thm}[Proposition 9.18 in \cite{abvNa}] \label{thm:prop_of_A_kappa}
	For every infinite cardinal $\kappa$, the von Neumann algebra $\mathcal A_\kappa$ (Definition \ref{def:a_kappa}) is abelian, $\sigma$-finite, diffuse, and $\kappa$-homogeneous.
\end{thm}

\begin{thm}[Theorem 9.20 in \cite{abvNa}] \label{thm:Maharam}
	Let $(A,\tau)$ be a diffuse abelian von Neumann measure algebra. Then there are families $(\kappa_i)_{i \in I}$ and $(\lambda_i)_{i \in I}$, where each $\kappa_i$ is an infinite cardinal and each $\lambda_i$ is a positive real number, such that
	\[
		(A,\tau) \cong \bigoplus_{i \in I} (\mathcal A_{\kappa_i}, \lambda_i \tau_{\kappa_i}).
	\]
\end{thm}

\begin{thm} \label{thm:abvNa}
	Let $(A,\tau)$ be a diffuse abelian tracial von Neumann algebra. Then $(A,\tau)$ is self-symmetric. 
\end{thm}

\begin{proof}
	Write $(A,\tau)$ in the sense of Theorem \ref{thm:Maharam} above. Since $\tau$ is a faithful normal tracial state, $I$ must be countable, and $\sum_{i \in I} \lambda_i = 1$. Fix $0 < t < 1$, we need to show there exists a projection $p \in A$ with $\tau(p) = t$ and
	\begin{equation} \label{eq:self_sym}
		(Ap, t^{-1}\tau|_{Ap}) \cong (A, \tau) \cong (A(1-p), (1-t)^{-1}\tau|_{A(1-p)}).
	\end{equation}
	
	For each $i \in I$, we take a projection $p_i \in \mathcal A_{\kappa_i}$ with $\tau_{\kappa_i}(p_i) = t$. Now note for all non-zero projections $q \in \mathcal A_{\kappa_i} p_i$, 
	\[
		\mathfrak{k}((\mathcal A_{\kappa_i} p_i)q) = \mathfrak{k}(\mathcal A_{\kappa_i}q) = \mathfrak{k}(\mathcal A_{\kappa_i}) = \mathfrak{k}(\mathcal A_{\kappa_i} p_i) = \kappa_i, 
	\]
	because $\mathcal A_{\kappa_i}$ is $\kappa_i$-homogeneous, and $q$ and $p_i$ are both non-zero projections in $\mathcal A_{\kappa_i}$. This implies $\mathcal A_{\kappa_i} p_i$ is $\kappa_i$-homogeneous, and $(\mathcal A_{\kappa_i} p_i, t^{-1} \lambda_i \tau_{\kappa_i}|_{\mathcal A_{\kappa_i} p_i})$ is a $\kappa_i$-homogeneous von Neumann measure algebra with 
	\[
		t^{-1} \lambda_i \tau_{\kappa_i}(p_i) = \lambda_i = \lambda_i \tau_{\kappa_i}(1_{\kappa_i}),
	\]
	where $p_i$ is the identity in $\mathcal A_{\kappa_i} p_i$, and $1_{\kappa_i}$ is the identity in $\mathcal A_{\kappa_i}$. Then by Theorem \ref{thm:k-homo_implies_isomo} and \ref{thm:prop_of_A_kappa}, we have a tracial isomorphism 
	\[
		(\mathcal A_{\kappa_i} p_i, t^{-1} \lambda_i \tau_{\kappa_i}|_{\mathcal A_{\kappa_i} p_i}) \cong (\mathcal A_{\kappa_i}, \lambda_i \tau_{\kappa_i}).
	\]
	
	Doing so for all $i \in I$, and letting $p = \sum_{i \in I} p_i$, we obtain in the direct sum
	\[
		(Ap, t^{-1}\tau|_{Ap}) \cong \bigoplus_{i \in I} (\mathcal A_{\kappa_i} p_i, \lambda_i t^{-1} \tau_{\kappa_i}) \cong \bigoplus_{i \in I} (\mathcal A_{\kappa_i}, \lambda_i \tau_{\kappa_i}) \cong (A, \tau),
	\]
	which proves the first isomorphism in (\ref{eq:self_sym}). Similarly, we can obtain the second isomorphism in (\ref{eq:self_sym}) and thus $(A,\tau)$ is self-symmetric. 
\end{proof}

\begin{rem}
	This theorem provides us with many examples of self-symmetric tracial von Neumann algebras, including the group von Neumann algebras of all infinite discrete abelian groups, regardless of the cardinality.
(See, for example, Proposition 5.1 of \cite{FDIM} for a proof of the well-known fact that all such algebras are diffuse.)

\end{rem}

\section{Well-definedness and properties of \texorpdfstring{$\mathcal{F}_{s,r}$}{F s,r}}

Throughout this section, unless otherwise noted, we fix a self-symmetric tracial von Neumann algebra $A$. We denote $\mathcal{F}_{n,r} = A^{*n} * L\mathbb{F}_r$ for $n \in \mathbb{N}$ and $r \in (1, \infty]$. Since the free product of any diffuse tracial von Neumann algebra and any non-trivial tracial von Neumann algebra is a II$_1$ factor (see Corollary 5.3.8 in \cite{II1} for example). Our plan is to define $\mathcal F_{s,r}$ for more general $s$ and $r$ by rescalings of these $\mathcal F_{n,r}$. We first compute $(\mathcal{F}_{n,r})^{1/k}$.

\begin{lemma} \label{lem:A_nk}
	Let $n, k \in \mathbb{N}$ and $r \in (1, \infty]$, then
	\[
		(A^{*n} * L\mathbb{F}_r)^{1/k} \cong A^{*(nk)} * L\mathbb{F}_{(n+r-1)k^2 - nk +1},
	\]
	in other words
	\[
		(\mathcal{F}_{n,r})^{1/k} \cong \mathcal{F}_{nk, (n+r-1)k^2 - nk +1}.
	\]
\end{lemma}

\begin{proof}
	We first prove the case $n=1$ and proceed by induction on $n$. From the self-symmetry of $A$, we take projections $p_1, \ldots, p_k \in A$, where $\sum_{j=1}^k p_j = 1_A$, $\tau(p_j)=1/k$, and $(Ap_j, k\tau|_{Ap_j}) \cong (A, \tau)$ for all $1 \leq j \leq k$. Then
	\begin{align*}
		(A * L\mathbb{F}_r)^{1/k} & \cong p_1((Ap_1 \oplus \cdots \oplus Ap_k) * L\mathbb{F}_r)p_1 \\
		& \cong Ap_1 * p_1((\mathbb{C}p_1 \oplus Ap_2 \oplus \cdots \oplus Ap_k) * L\mathbb{F}_r)p_1 \\
		& \cong Ap_1 * p_2((\mathbb{C}p_1 \oplus Ap_2 \oplus \cdots \oplus Ap_k) * L\mathbb{F}_r)p_2 \\
		& \cong \cdots \\
		& \cong Ap_1 * \cdots Ap_k * ((\mathbb Cp_1 \oplus \cdots \oplus \mathbb Cp_k) * L\mathbb F_r)^{1/k} \\
		& \cong A^{*k} * (L\mathbb{F}_{\frac{k(k-1)}{k^2}+r})^{1/k} \cong A^{*k} * L\mathbb{F}_{rk^2 - k + 1}.
	\end{align*}
follows from Claim 4.3 in \cite{NON_ISOM} and Proposition 2.4 in \cite{FDIM}. Next suppose the lemma holds for some $n \in \mathbb{N}$. Then 
	\begin{align*}
		(A^{*(n+1)} * L\mathbb{F}_r)^{1/k} & = (A * A^{*n} * L\mathbb{F}_r)^{1/k} \\
		& \cong p_1((Ap_1 \oplus \cdots \oplus Ap_k) * A^{*n} * L\mathbb{F}_r)p_1 \\
		& \cong Ap_1 * \cdots * Ap_k * p_1((\mathbb{C}p_1 \oplus \cdots \oplus \mathbb{C}p_k) * A^{*n} * L\mathbb{F}_r)p_1 \\
		& \cong A^{*k} * (A^{*n} * L\mathbb{F}_{\frac{k(k-1)}{k^2}+r})^{1/k} \\
		& \cong A^{*k} * (A^{*(nk)} * L\mathbb{F}_{(n+r)k^2 - (n+1)k + 1}) \\
		& \cong A^{*(n+1)k} * L\mathbb{F}_{(n+r)k^2 - (n+1)k + 1}
	\end{align*}
	follows similarly by associativity and commutativity of the free product, and the inductive hypothesis. 
	
\end{proof}

The next lemma is key to well-definedness:

\begin{lemma} \label{lem:isom_compre_cond}
	Let $n,m \in \mathbb{N}$, $r, u \in (1, \infty]$, and $t,s \in (0, \infty)$ satisfy
	\[
		\frac{n}{t} = \frac{m}{s} \quad \text{ and } \quad \frac{n+r-1}{t^2} = \frac{m+u-1}{s^2}
	\]
	then
	\[
		(\mathcal{F}_{n,r})^t \cong (\mathcal{F}_{m,u})^s.
	\]
\end{lemma}

\begin{proof}
	Using the previous lemma, we can compute
	\[
		((\mathcal{F}_{n,r})^t)^{1/m} \cong ((\mathcal{F}_{n,r})^{1/m})^{t} \cong (\mathcal{F}_{nm,(n+r-1)m^2-nm+1})^t
	\]
	and
	\[
		((\mathcal{F}_{m,u})^s)^{1/m} \cong ((\mathcal{F}_{m,u})^{1/n})^{t} \cong (\mathcal{F}_{nm,(m+u-1)n^2-nm+1})^t.
	\]
	The right-hand sides are precisely the same by our assumption, so
	\[
		(\mathcal{F}_{n,r})^t \cong M_m(\mathbb{C}) \otimes (\mathcal{F}_{nm,(m+u-1)n^2-nm+1})^t \cong (\mathcal{F}_{m,u})^s. 
	\] 
\end{proof}

\begin{defi} \label{def:IFGF}
	For a self-symmetric tracial von Neumann algebra $A$, for $s \in (0, \infty)$, and $r \in (1-s, \infty]$, we choose $n \in \mathbb{N}$ large enough so that 
	\[
		\frac{s+r-1}{s^2}n^2 - n + 1 > 1
	\]
	and define
	\begin{equation} \label{eq:F_def}
		\mathcal{F}_{s,r}(A) := (A^{*n} * L\mathbb{F}_{\frac{s+r-1}{s^2}n^2 - n + 1})^{n/s}.
	\end{equation}
	When the algebra $A$ is clear from context, we abbreviate the above as $\mathcal{F}_{s,r}$. When $r = \infty$, we use its usual arithmetic in the extended real number system. 
\end{defi}

\begin{rem} \label{rem:def_IFGF}
	This definition gives isomorphic copies of II$_1$ factors regardless of $n$ by Lemma \ref{lem:isom_compre_cond}. It is also elementary to verify that this definition recovers $\mathcal{F}_{s,r} \cong A^{*s} * L\mathbb{F}_{r}$ when $s \in \mathbb{N}$ and $r \in (1, \infty]$, and $\mathcal{F}_{s,0} \cong A^{*s}$ when $s \in \mathbb{N}$ and $s \geq 2$. The first isomorphism is obtained by setting $n = s$ in (\ref{eq:F_def}). The second isomorphism follows from setting $n = 2s$ in (\ref{eq:F_def}), and then taking the $1/2$ compression:
	\[
		(\mathcal F_{s,0})^{1/2} \cong A^{*(2s)} * L\mathbb F_{2s-3} \cong (A^{*s})^{1/2},
	\]
	which follows from (\ref{eq:A_*n^t}).
\end{rem}

The rescaling formula (\ref{eq:F_rescale}) now follows directly from well-definedness. 

\begin{thm}\label{thm:F_rescale}
	Let $s \in (0, \infty)$, $r \in (1-s, \infty]$, and $t \in (0, \infty)$. Then
	\[
 		(\mathcal{F}_{s,r})^t \cong \mathcal{F}_{\frac{s}{t},\frac{s+r-1}{t^2} - \frac{s}{t} + 1}.
 	\]
\end{thm}

\begin{proof}
	It suffices to choose $n \in \mathbb{N}$ with 
	\[
		\frac{s+r-1}{s^2}n^2 - n + 1 > 1
	\]  
	to realize both sides of the isomorphism above as in (\ref{eq:F_def}). Then it follows 
	\begin{align*}
		\mathcal{F}_{\frac{s}{t},\frac{s+r-1}{t^2} - \frac{s}{t} + 1} & \cong (A^{*n} * L\mathbb{F}_{\frac{s+r-1}{s^2}n^2 - n + 1})^{nt/s} \\
		& \cong (\mathcal{F}_{s,r})^t.
	\end{align*}
\end{proof}

\begin{rem}
	The rescaling formula in particular extends the classical rescaling formula of the interpolated free group factors when $A$ is separable, and the result in Corollary 4.4(1) of \cite{NON_ISOM} when $s \geq 2$ is a natural number, $r=0$, and $t = 1/m$ for some $m \in \mathbb{N}$.
\end{rem}

\begin{thm}
	Let $A$ be a non-separable self-symmetric abelian tracial von Neumann algebra. Let $0 < s \neq s' < \infty$, $r \in (1-s, \infty]$, and $r' \in (1-s', \infty]$. Then
	\[
		\mathcal{F}_{s,r}(A) \not \cong \mathcal{F}_{s',r'}(A).
	\]
\end{thm}

\begin{proof}
	We use the sans-rank of \cite{NON_ISOM} to distinguish.
	From Remark 2.11, Proposition 2.12, Theorem 3.8, and proof of Theorem 1.1 of \cite{NON_ISOM}, 
	\[
		\rankns(\mathcal{F}_{s,r}(A)) = s \cdot \rankns(A) = s \cdot \tau(p) \neq s' \cdot \tau(p) = \rankns(\mathcal{F}_{s',r'}(A)) ,
	\] 
	where $p \in A$ is the maximal, necessarily non-zero projection such that $Ap$ is purely non-separable. Therefore, the non-isomorphism holds regardless of $r$ and $r'$.
\end{proof}

Now we work on proving the addition formula (\ref{eq:F_addition}).

\begin{lemma} \label{lem:F_s+n,r}
	Let $n \in \mathbb{N}$ and $s \in (0, \infty)$ satisfy
	\begin{enumerate}[(i)]
		\item $n \geq 3$,
		\item $\max \{ t, 1-t \} < \frac{n}{n+1}$,
		\item $r = \frac{n-1}{t^2} - \frac{n}{t} + 1$,
	\end{enumerate}
	where $t = \frac{n}{s+n}$. Then we have
	\[
		A^{*n} * \mathcal{F}_{s,r} \cong \mathcal{F}_{s+n,r}.
	\]
\end{lemma}

\begin{proof}
	Denote the $n$ copies of $A$ in $A^{*n}$ by $A_1, \ldots, A_n$ respectively. For all $1 \leq j \leq n$, by self-symmetry of $A_j$, we take a projection $e_j \in A_j$ of trace $t$ with the trace-preserving isomorphisms as in Definition \ref{def:self_symmetric}. We know $e_1 \sim \cdots \sim e_n$ in 
	\[
		(\overset{e_1}{\underset{t}{\mathbb{C}}} \oplus \overset{1-e_1}{\underset{1-t}{\mathbb{C}}}) * \cdots * (\overset{e_n}{\underset{t}{\mathbb{C}}} \oplus \overset{1-e_n}{\underset{1-t}{\mathbb{C}}}),
	\]
	since, using conditions (i) and (ii), the above is known to be an interpolated free group factor by Theorem 1.1 and 4.6 in \cite{FDIM}. From Fact 4.2 of \cite{NON_ISOM}, 
	\begin{align*}
		(A_1 * \cdots & * A_n)^t \\
		& \cong e_1((A_1e_1 \oplus A_1(1-e_1)) * A_2 * \cdots * A_n)e_1 \\
		& \cong A_1e_1 * e_1((\mathbb{C}e_1 \oplus A_1(1-e_1)) * A_2 * \cdots * A_n)e_1 \\
		& \cong A_1e_1 * e_2((\overset{e_1}{\underset{t}{\mathbb{C}}} \oplus \overset{1-e_1}{\underset{1-t}{A_1}}) * (A_2e_2 \oplus A_2(1-e_2)) * \cdots * A_n)e_2 \\
		& \cong \cdots \\
		& \cong A_1e_1 * \cdots * A_ne_n * e_n((\overset{e_1}{\underset{t}{\mathbb{C}}} \oplus \overset{1-e_1}{\underset{1-t}{A_1}}) * \cdots * (\overset{e_n}{\underset{t}{\mathbb{C}}} \oplus \overset{1-e_n}{\underset{1-t}{A_n}}))e_n \\
		& \cong A^{*n} * ((\underset{t}{\mathbb{C}} \oplus \underset{1-t}{A})^{*n})^t.
	\end{align*}
	In the same way, using Theorem 1.1 and 4.6 of \cite{FDIM},
	\begin{align*}
		((\underset{t}{\mathbb{C}} \oplus \underset{1-t}{A})^{*n})^{1-t} & \cong  A^{*n} * ((\underset{t}{\mathbb{C}} \oplus \underset{1-t}{\mathbb{C}})^{*n})^{1-t} \\
		& \cong A^{*n} * (L\mathbb{F}_{2nt(1-t)})^{1-t} .
	\end{align*}
	Employing Definition \ref{def:IFGF} and formula (\ref{eq:F_rescale}), we obtain:
	\begin{align*}
		 \mathcal{F}_{\frac{n}{t}, \frac{n-1}{t^2} - \frac{n}{t} + 1} \cong (\mathcal F_{n,0})^t & \cong (A^{*n})^t \\
		 & \cong A^{*n} * (A^{*n} * (L\mathbb F_{2nt(1-t)})^{1-t})^{\frac{t}{1-t}} \\
		 & \cong A^{*n} * (\mathcal{F}_{n, 1 + \frac{2nt(1-t)-1}{(1-t)^2}})^{\frac{t}{1-t}} \\
		 & \cong A^{*n} * \mathcal{F}_{\frac{n}{t}-n, \frac{n-1}{t^2} - \frac{n}{t} + 1}.
	\end{align*}
	which is exactly
	\[
		A^{*n} * \mathcal{F}_{s,r} \cong \mathcal{F}_{s+n,r}.
	\]
	
\end{proof}

\begin{lemma} \label{lem:F_s,r+u}
	Let $s \in (0, \infty)$, $r \in (1-s, \infty]$, and $u \in (1, \infty]$. Then 
	\[
		\mathcal{F}_{s,r} * L\mathbb{F}_u \cong \mathcal{F}_{s,r+u}.
	\]
\end{lemma}

\begin{proof}
	We choose $n \in \mathbb{N}$ large enough so that $\frac{n^2}{s^2} - 1 > 1$ and $\mathcal{F}_{s,r} \cong (\mathcal{F}_{n,\frac{r+s-1}{s^2}n^2 - n + 1})^{n/s}$ as in Definition \ref{def:IFGF}. It then follows from Theorem 1.5 in \cite{DR00} that
	\begin{align*}
		& (\mathcal{F}_{s,r} * L\mathbb{F}_u)^{s/n} \cong (\mathcal{F}_{s,r})^{s/n} * (L\mathbb{F}_u)^{s/n} * L\mathbb{F}_{\frac{n^2}{s^2} - 1} \\
		& \cong A^{*n} * L\mathbb{F}_{\frac{r+s-1}{s^2}n^2 - n + 1} * L\mathbb{F}_{1+(u-1)\frac{n^2}{s^2}} * L\mathbb{F}_{\frac{n^2}{s^2} - 1} \\
		& \cong A^{*n} * L\mathbb{F}_{\frac{r+s-1}{s^2}n^2 - n + 1 + u\frac{n^2}{s^2}}.
	\end{align*}
	Hence, by the rescaling formula (\ref{eq:F_rescale}),
	\[
		\mathcal{F}_{s,r} * L\mathbb{F}_u \cong (\mathcal{F}_{n, \frac{r+s-1}{s^2}n^2 - n + 1 + u\frac{n^2}{s^2}})^{n/s} \cong \mathcal{F}_{s, r+u}.
	\]
\end{proof}

\begin{thm} \label{thm:addition_formula}
	Let $s, v \in (0, \infty)$, $r \in (1-s, \infty]$, $u \in (1-v, \infty]$, then
	\[
		\mathcal{F}_{s,r} * \mathcal{F}_{v,u} \cong \mathcal{F}_{s+v, r+u}.
	\]
\end{thm}

\begin{proof}
	Without loss of generality we assume $v \geq s$. We realize $\mathcal{F}_{s,r}$ and $\mathcal{F}_{v,u}$ by choosing $n,m \in \mathbb{N}$ large enough so that 
	\begin{enumerate}[(i)]
		\item \[ n^2/s^2 - 1 > 1, \]
		\item \[ n \geq 3, \]
		\item \[ \frac{v}{s+v} < \frac{n}{n+1}, \]
		\item \[ \frac{r+s+u+v-1}{s^2}n^2 - \frac{(s+v)^2}{s^2}n + \frac{(s+v)^2}{s^2} > 1, \]
	\end{enumerate}
	 as well as \[\mathcal{F}_{s,r} \cong (\mathcal{F}_{n,\frac{r+s-1}{s^2}n^2 - n + 1})^{n/s}\] and \[\mathcal{F}_{v,u} \cong (\mathcal{F}_{m,\frac{u+v-1}{v^2}m^2 - m + 1})^{m/v}\] as in Definition \ref{def:IFGF}. We let $t = s/n$. It follows from Theorem 1.5 in \cite{DR00} using condition (i), rescaling formula (\ref{eq:F_rescale}), and Lemma \ref{lem:F_s,r+u} that 
	\begin{align*}
		(\mathcal{F}_{s,r} * \mathcal{F}_{v,u})^t & \cong (\mathcal{F}_{n,\frac{r+s-1}{s^2}n^2 - n + 1})^{nt/s} * (\mathcal{F}_{m,\frac{u+v-1}{v^2}m^2 - m + 1})^{mt/v} * L\mathbb{F}_{t^{-2}-1} \\
		& \cong \mathcal{F}_{n,\frac{r+s-1}{s^2}n^2 - n + 1} * \mathcal{F}_{\frac{v}{s}n, \frac{u+v-1}{s^2}n^2 - \frac{v}{s}n + 1} * L\mathbb{F}_{t^{-2}-1} \\
		& \cong \mathcal{F}_{n,0} * L\mathbb{F}_{\frac{r+s-1}{s^2}n^2 - n + 1} * \mathcal{F}_{\frac{v}{s}n, \frac{u+v-1}{s^2}n^2 - \frac{v}{s}n + 1} * L\mathbb{F}_{t^{-2}-1} \\
		& \cong A^{*n} * \mathcal{F}_{\frac{v}{s}n,\frac{r+s+u+v-1}{s^2}n^2 - \frac{s+v}{s}n + 1}.
	\end{align*}

	By Lemma \ref{lem:F_s+n,r} using conditions (ii) and (iii), $A^{*n}$ can be absorbed like so:
	\begin{align*}
		A^{*n} * \mathcal{F}_{\frac{v}{s}n,(\frac{(s+v)^2}{s^2} - \frac{s+v}{s}) n - \frac{(s+v)^2}{s^2} +1} \cong \mathcal{F}_{\frac{v}{s}n + n,(\frac{(s+v)^2}{s^2} - \frac{s+v}{s}) n - \frac{(s+v)^2}{s^2} +1}.
	\end{align*}
	Now by Lemma \ref{lem:F_s,r+u} using condition (iv), we add back an interpolated free group factor to obtain: 
	\begin{align*}
		(\mathcal{F}_{s,r} & * \mathcal{F}_{v,u})^t \\
		& \cong A^{*n} * \mathcal{F}_{\frac{v}{s}n,\frac{r+s+u+v-1}{s^2}n^2 - \frac{s+v}{s}n + 1} \\
		& \cong A^{*n} * \mathcal{F}_{\frac{v}{s}n,(\frac{(s+v)^2}{s^2} - \frac{s+v}{s}) n - \frac{(s+v)^2}{s^2} +1} * L\mathbb{F}_{\frac{r+s+u+v-1}{s^2}n^2 - \frac{(s+v)^2}{s^2}n + \frac{(s+v)^2}{s^2}} \\
		& \cong \mathcal{F}_{\frac{v}{s}n + n,(\frac{(s+v)^2}{s^2} - \frac{s+v}{s}) n - \frac{(s+v)^2}{s^2} +1} * L\mathbb{F}_{\frac{r+s+u+v-1}{s^2}n^2 - \frac{(s+v)^2}{s^2}n + \frac{(s+v)^2}{s^2}} \\
		& \cong \mathcal{F}_{\frac{v}{s}n + n, \frac{r+s+u+v-1}{s^2}n^2 - \frac{s+v}{s}n + 1} \\
		& \cong (\mathcal{F}_{s+v, r+u})^t,
	\end{align*}
	where the last isomorphism is just the rescaling formula (\ref{eq:F_rescale}). The theorem is then proved by an amplification by $1/t$ on both sides.
	
\end{proof}

\section{More results and applications}


For a non-separable self-symmetric tracial von Neumann algebra $A$, heuristically, we can think of the parameter $r$ in $\mathcal{F}_{s,r}$ as the free dimension of the separable component, and the parameter $s$ as the free dimension of the non-separable component. The next proposition shows different von Neumann algebras with the same free dimension would give isomorphic free products with our family $\mathcal{F}_{s,r}$. It is therefore an extension of Lemma \ref{lem:F_s,r+u} above.

\begin{prop} \label{prop:F_s,r+fdim}
	Fix some $\mathcal{F}_{s,r}$ and $u \in (0, \infty]$. Let $B_u$ be any finite-dimensional (with dimension $\geq 2$) or hyperfinite von Neumann algebra, or an interpolated free group factor, or direct sums of those, with free dimension $u$. Then
	\[
		\mathcal{F}_{s,r} * B_u \cong \mathcal{F}_{s,r+u}.
	\]
\end{prop}

\begin{proof}
	We choose $n \in \mathbb{N}$ large enough so that $B_u * M_n(\mathbb{C})$ is a factor. By Theorem 4.6 of \cite{FDIM}, it suffices for the minimal projections in $B_u$ to have trace less than $1 - n^{-2}$, and then $B_u * M_n(\mathbb{C}) \cong L\mathbb{F}_{u + 1 - n^{-2}}$. We have by Theorem 1.2 of \cite{FDIM}
	\begin{align*}
		(\mathcal{F}_{s,r} * B_u)^{1/n} & \cong ((M_n(\mathbb{C}) \otimes (\mathcal{F}_{s,r})^{1/n}) * B_u)^{1/n} \\
		& \cong (\mathcal{F}_{s,r})^{1/n} * (B_u * M_n(\mathbb{C}))^{1/n}.
	\end{align*}
	Using the rescaling formula (\ref{eq:F_rescale}) and Lemma \ref{lem:F_s,r+u}:
	\begin{align*}
		(\mathcal{F}_{s,r} * B_u)^{1/n} & \cong \mathcal{F}_{ns, n^2(s+r-1)-ns+1} * (L\mathbb{F}_{u+1-n^{-2}})^{1/n} \\
		& \cong \mathcal{F}_{ns, n^2(s+r-1)-ns+1} * L\mathbb{F}_{1 + n^2(u-n^{-2})} \\
		& \cong \mathcal{F}_{ns, n^2(u+s+r-1)-ns+1} \\
		& \cong (\mathcal{F}_{s, r+u})^{1/n}.
	\end{align*}
	Now our proposition follows by tensoring with $M_n(\mathbb{C})$.
\end{proof}

The next proposition collects some more free product results involving the family $\mathcal{F}_{s,r}$.

\begin{prop} \label{prop:misc_results}
	For a self-symmetric tracial von Neumann algebra $A$, $n,k \in \mathbb{N}$, $r \in (1, \infty]$ and $t \in (0,1)$, we have
	\begin{enumerate}[(i)]
		\item\label{it:misc-i} \[ A * L\mathbb{Z} \cong (\mathcal{F}_{k, k^2-k+1})^k \cong A * \mathcal{R}, \]
		\item\label{it:misc-ii} \[ A^{*n} * (\underset{t}{A} \oplus \underset{1-t}{\mathbb{C}}) \cong \mathcal{F}_{n+t, t-t^2}, \]
		\item\label{it:misc-iii} \[ (A \otimes M_k(\mathbb{C})) * L\mathbb{F}_r \cong \mathcal{F}_{\frac{1}{k}, r-\frac{1}{k}+1}, \]
\item\label{it:misc-new} 
\[
(\smd At \oplus \smd\Cpx{1-t} )*L(\mathbb F_r)\cong \Fc_{t,r+t-t^2},
\]
		\item\label{it:misc-iv} \[ (\underset{t}{A} \oplus \underset{1-t}{L\mathbb{Z}})^{*n} \cong \mathcal{F}_{nt, n(1-t)}. \]
	\end{enumerate}
\end{prop}

\begin{proof}
    We first comment that~\eqref{it:misc-i} is analogous to Theorem 4.1 of \cite{MATMOD}, and~\eqref{it:misc-ii} is analogous to Remark 1.5 of \cite{FDIM}. 
    
     For~\eqref{it:misc-i}, it suffices to compute the $1/k$ compression of $A * L\mathbb Z$ for $k \in \mathbb N$.  As in the beginning of the proof of Lemma \ref{lem:A_nk}, we choose projections $p_1, \ldots, p_k $ in $A$ to obtain
     \begin{align*}
     	(A * L \mathbb Z)^{1/k} & \cong ((Ap_1 \oplus \cdots \oplus Ap_k) * L\mathbb Z )^{1/k} \\
     	& \cong Ap_1 * \cdots * Ap_k * ((\smdp{\Cpx}{1/k}{p_1} \oplus \cdots \oplus \smdp{\Cpx}{1/k}{p_k} ) * L\mathbb Z )^{1/k} \\
     	& \cong A^{*k} * (L\mathbb F_{\frac{k(k-1)}{k^2}+1})^{1/k} \\
     	& \cong A^{*k} * L\mathbb F_{k^2-k+1} \\
     	& \cong \mathcal F_{k, k^2-k+1}.
     \end{align*}
     The computation with $\mathcal R$ in place of $L \mathbb Z$ is exactly the same because both have free dimension $1$.
     
     For~\eqref{it:misc-ii}, it follows from the same argument as in the proof of Lemma \ref{lem:F_s+n,r}. For~\eqref{it:misc-iii}, it is a direct application of Theorem 1.2 of \cite{FDIM}.

For~\eqref{it:misc-new}, let
\[
\Mcal=(\smdp Atp\oplus\smd\Cpx{1-t})*L(\mathbb F_r).
\]
By Theorem~1.2 of~\cite{FDIM}, 
\[
\Mcal^t\cong p\Mcal p
\cong A*\left((\smd \Cpx t\oplus\smd\Cpx{1-t})*L(\mathbb F_r)\right)^t
\cong A* L(\mathbb F_{r+2t(1-t)})^t
\cong \Fc_{1,\frac{r-1}{t^2}+\frac2t-1}.
\]
Thus, 
\[
\Mcal\cong\Fc_{1,\frac{r-1}{t^2}+\frac2t-1}^{1/t}\cong\Fc_{t,r+t-t^2}.
\]
     
     For~\eqref{it:misc-iv}, we choose $k \in \mathbb{N}$ large enough so that
     \[
     	(\underset{t/k}{\mathbb{C}} \oplus \cdots \oplus \underset{t/k}{\mathbb{C}} \oplus \underset{1-t}{L\mathbb{Z}})^{*n}
     \]
     is a factor. Then the same argument as in Lemma \ref{lem:F_s+n,r} shows
     \begin{align*}
     	& ((\underset{t}{A} \oplus \underset{1-t}{L\mathbb{Z}})^{*n})^{t/k} \cong A^{*(nk)} * ((\underset{t/k}{\mathbb{C}} \oplus \cdots \oplus \underset{t/k}{\mathbb{C}} \oplus \underset{1-t}{L\mathbb{Z}})^{*n})^{t/k} \\
     	& \cong \mathcal{F}_{nk, (n-1)\frac{k^2}{t^2}-nk+1}.
     \end{align*}
     Rescaling by $k/t$ gives~\eqref{it:misc-iv}. 
\end{proof}

\begin{rem}
	Note that for all $t \in (0,1)$, \[A_t := \underset{t}{A} \oplus \underset{1-t}{L\mathbb{Z}}\] is still a non-separable self-symmetric abelian tracial von Neumann algebra if $A$ is so, and $A_t \not \cong A$ since $\rankns(A_t) = t \cdot \rankns(A)$. For all $s \in (0,\infty)$ and $r \in (1-s, \infty]$, we have
	\[
		\mathcal{F}_{s,r}(A_t) \cong \mathcal{F}_{st, s+r-st}(A)
	\]
	by~\eqref{it:misc-iv} above. 

\end{rem}

We now consider countably infinite free products of $\mathcal F_{s_i,r_i}(A)$ for a self-symmetric tracial von Neumann algebra $A$.
(Note:  the result below is generalized in Theorem~\ref{thm:finitessum}.)

\begin{thm} \label{thm:A_infty_free_prod}
	Fix a self-symmetric tracial von Neumann algebra $A$. Assume for each $i \in \mathbb N$, $s_i \in (0, \infty)$, $r_i \in (1-s_i, \infty]$, and 
	\[
		\sum_{i \in \mathbb N} s_i = \infty.
	\]
	Then 
	\[
		\underset{i \in \mathbb N}{\bigfreeprod} \mathcal F_{s_i,r_i} \cong A^{*\infty}.
	\]
\end{thm}

\begin{proof}
	From Theorem 1.5 of \cite{DR00} and Proposition \ref{prop:F_s,r+fdim}, we have
	\[
		\underset{i \in \mathbb N}{\bigfreeprod} \mathcal F_{s_i,r_i} \cong \left ( \underset{i \in \mathbb N}{\bigfreeprod} \mathcal F_{s_i,r_i} \right ) * L\mathbb F_\infty \cong \underset{i \in \mathbb N}{\bigfreeprod} (\mathcal F_{s_i,r_i} * L\mathbb F_\infty ) \cong \underset{i \in \mathbb N}{\bigfreeprod} \mathcal F_{s_i,\infty},
	\]
	so we may without loss of generality assume $r_i = \infty$ for all $i \in \mathbb N$. From the free product addition formula (\ref{eq:F_addition}), we have for $l \in \mathbb N$,
	\[
		\underset{i=k}{\overset{k+l}{\bigfreeprod}} \mathcal F_{s_i, \infty} \cong \mathcal F_{s, \infty},
	\]
	where $s = s_k + \cdots + s_{k+l} $, so we may without loss of generality assume $s_i \geq 3$ for all $i \in \mathbb N$ by regrouping the infinite free product, if necessary.
Moreover, from Proposition \ref{prop:misc_results}, we have
	\begin{equation} \label{eq:F_s_t_inf}
		\mathcal F_{s,\infty} * ( \underset{t}{A} \oplus \underset{1-t}{\mathbb C} ) \cong \mathcal F_{s+t, \infty}
	\end{equation}
	for all $s > 1$ and $0 \leq t < 1$, where we abuse notation to mean
	\[
		\underset{0}{A} \oplus \underset{1}{\mathbb C} \cong \mathbb C.
	\]
	
	Now let
	\[
		n_1=\lfloor s_i \rfloor,\quad t_1=s_1-n_1
	\]
	and recursively choose, for all $i \ge 2$,
	\[
		n_i=\lfloor s_i-(1-t_{i-1}) \rfloor,\quad t_i=s_i-(1-t_{i-1})-n_i.
	\]
	In particular, we have $n_1\ge1$, $0\le t_1<1$, and
	\[
		\mathcal F_{s_1, \infty} \cong A^{*n_1} * L\mathbb F_\infty * ( \underset{t_1}{A} \oplus \underset{1-t_1}{\mathbb C} ),
	\]
	which follows from Remark \ref{rem:def_IFGF} and (\ref{eq:F_s_t_inf}).
	For all $i\ge2$, we have $n_i\ge1$, $t_i\in[0,1)$ and $s_i=(1-t_{i-1})+n_i+t_i$, so
	\[
		\mathcal F_{s_i, \infty} \cong ( \underset{1-t_{i-1}}{A} \oplus \underset{t_{i-1}}{\mathbb C} ) * A^{*n_i} * L\mathbb F_\infty * ( \underset{t_i}{A} \oplus \underset{1-t_i}{\mathbb C} ) .
	\]
	Substituting these into the infinite free product and using associativity and commutativity of the free product, we get
	\begin{align*}
		\overset{\infty}{\underset{i=1}{\bigfreeprod}} & \mathcal F_{s_i,\infty} \\
		& \cong A^{*n_1} * L\mathbb F_\infty * ( \underset{t_1}{A} \oplus \underset{1-t_1}{\mathbb C} ) * \left( \overset{\infty}{\underset{i=2}{\bigfreeprod}} ( \underset{1-t_{i-1}}{A} \oplus \underset{t_{i-1}}{\mathbb C} ) * A^{*n_i} * L\mathbb F_\infty * ( \underset{t_i}{A} \oplus \underset{1-t_i}{\mathbb C} ) \right) \\
		& \cong A^{*n_1} * L\mathbb F_\infty * \left( \overset{\infty}{\underset{i=1}{\bigfreeprod}} 
		\left( 
			( \underset{t_i}{A} \oplus \underset{1-t_i}{\mathbb C} ) * ( \underset{1-t_i}{A} \oplus \underset{t_i}{\mathbb C} ) * A^{*n_{i+1}} * L\mathbb F_\infty
		\right)\right) \\
		& \cong A^{*n_1} * L\mathbb F_\infty * 
		\left( 
			\overset{\infty}{\underset{i=1}{\bigfreeprod}} \left( A^{*(1+n_{i+1})}*L\mathbb F_\infty \right)
		\right) \\
		& \cong A^{*\infty} * L\mathbb F_\infty \\
		& \cong A^{*\infty},
	\end{align*}
	where we have used ~\eqref{eq:F_s_t_inf} twice to obtain the third isomorphism above. 
\end{proof}

\begin{thm}
	The Murray-von Neumann fundamental group of $A^{*\infty}$ is $\mathbb{R}_+^*$ for all self-symmetric tracial von Neumann algebras $A$.
\end{thm} 

\begin{proof}
	Let $t \in (0,1)$ be arbitrary. Then by Theorem 1.5 of \cite{DR00}, the rescaling formula (\ref{eq:F_rescale}), and the above Theorem \ref{thm:A_infty_free_prod},
	\begin{align*}
		& (A^{*\infty})^t \cong \underset{n \in \mathbb{N}}{\bigfreeprod} (A*A)^t \cong \underset{n \in \mathbb{N}}{\bigfreeprod} \mathcal{F}_{\frac{2}{t}, \frac{1}{t^2} - \frac{2}{t} + 1} \cong A^{*\infty}.
	\end{align*}
	This shows the fundamental group of $A^{*\infty}$ contains $(0,1)$ so is all of $\mathbb{R}_+^*$. 
\end{proof}

The remainder of this section is dedicated to proving Theorem~\ref{thm:finitessum}, which is a generalization of
Theorem~\ref{thm:A_infty_free_prod}.
We nonetheless keep the proof of Theorem~\ref{thm:A_infty_free_prod} presented above in our exposition,
because it is considerably easier than the proof of the more general result, and because Theorem~\ref{thm:finitessum}
is not needed for the above result about the fundamental group.

The following is well known, but for convenience we provide a proof.
\begin{lemma}\label{lem:freeproj}
Let $(\mathcal{M},\tau)$ be a tracial von Neumann algebra.
Suppose $p$ and $q$ are projections in $\mathcal{M}$ that are free and satisfy $\tau(p)=\tau(q)$.
Then there exists a unitary element, $u$, in the von Neumann algebra generated by $\{p,q,1\}$, such that $u^*qu=p$.
Moreover, the central support of $p$ in $\Mcal$ equals the central support of $q$ in $\Mcal$.
\end{lemma}
\begin{proof}
Let $\alpha=\tau(p)$.
If $\alpha=0$ or $\alpha=1$, then we may take $u=1$.
Suppose $0<\alpha<1$.
Let $\mu$ be the distribution of $pqp$.
Then $\mu$ is equal to the multiplicative free convolution of the measure $(1-\alpha)\delta_0+\alpha\delta_1$ with itself.
By Example~2.8 in~\cite{Voi87}, $\mu$ has an atom at $0$ with mass equal to $1-\alpha$.
Consequently, the support projection of $pqp$ is equal to $p$.
Let $qp=v(pqp)^{1/2}$ be the polar decomposition of $qp$.
Then $v$ is a partial isometry satisfying $v^*v=p$ and $vv^*=q$.
Similarly, from the polar decomposition of $(1-q)(1-p)$, we obtain a partial isometry $w$ satisfying $w^*w=1-p$ and $ww^*=1-q$.
Let $u=v+w$.
Then $u^*qu=p$.

The existence of the unitary $u$ implies the statement about the central supports.
\end{proof}

\begin{lemma}\label{lem:Cpx+B}
Let $B$ be a tracial von Neumann algebra and let $\alpha\in(0,1)$.
Let
\[
\Mcal=(\smdp\Cpx\alpha p\oplus B)*(\smdp\Cpx\alpha q\oplus\Bt),
\]
where $\Bt$ is a copy of $B$.
Consider the von Neumann subalgebras
\begin{align*}
\Nc&=(\smdp\Cpx\alpha p\oplus B)*(\smdp\Cpx\alpha q\oplus\Cpx),\\[1ex]
\Nct&=(\smdp\Cpx\alpha p\oplus\Cpx)*(\smdp\Cpx\alpha q\oplus\Bt)
\end{align*}
of $\Mcal$.
Then there is a trace-preserving automorphism $\beta$ of $\Mcal$ satisfying
\[
\beta(p)=p,\qquad\beta(\Nc)=\Nct\quad\text{ and }\quad\beta(\Nct)=\Nc.
\]
\end{lemma}
\begin{proof}
Let $\sigma$ be the free flip automorphism of $\Mcal$, namely, that interchanges the two copies $\Cpx\oplus B$ and $\Cpx\oplus\Bt$, mapping each identically onto the other.
Note that $\sigma$ is trace-preserving, $\sigma^2=\id$, $\sigma(p)=q$ and $\sigma(\Nc)=\Nct$.
Let $u\in W^*(\{p,q,1\})$ be the unitary from Lemma~\ref{lem:freeproj} satisfying $u^*qu=p$.
Note that we have $u\in\Nc\cap\Nct$.
Let $\beta(\cdot)=u^*\sigma(\cdot)u$.
Then $\beta$ is a trace-preserving automorphism of $\Mcal$ and
\begin{align*}
\beta(p)&=u^*qu=p, \\
\beta(\Nc)&=u^*\Nct u=\Nct \\
\beta(\Nct)&=u^*\Nc u=\Nc.
\end{align*}
\end{proof}

\begin{lemma}\label{lem:A+B}
Let $A$ and $B$ be tracial von Neumann algebras and let $\alpha\in(0,1)$.
Let
\[
(\Mcal,\tau_\Mcal)=(\smdp A\alpha p\oplus B)*(\smdp\Cpx\alpha q\oplus\Bt),
\]
where $\Bt$ is a copy of $B$.
Consider the von Neumann subalgebras
\begin{align*}
\Nc&=(\smdp\Cpx\alpha p\oplus B)*(\smdp\Cpx\alpha q\oplus\Cpx),\\[1ex]
\Nct&=(\smdp\Cpx\alpha p\oplus\Cpx)*(\smdp\Cpx\alpha q\oplus\Bt).
\end{align*}
Then there is a trace-preserving automorphism $\gamma$ of $\Mcal$ such that
\[
\gamma(p)=p,\qquad\gamma(\Nc)=\Nct,\qquad\gamma(\Nct)=\Nc
\]
and $\gamma$ maps the subalgebra $A\oplus\Cpx$ of $\Mcal$ identically to itself (by which we mean the restriction of $\gamma$ to this subalgebra
is the identity map of the subalgebra to itself).
\end{lemma}
\begin{proof}
Consider the von Neumann subalgebra
\[
\Qc=(\smdp\Cpx\alpha p\oplus B)*(\smdp\Cpx\alpha q\oplus\Bt)
\]
of $\Mcal$.
Clearly, $A=A\oplus0$ and $\Qc$ together generate $\Mcal$.
By Theorem~1.2 of~\cite{FDIM}, $A$ and $p\Qc p$ are free in $p\Mcal p$ with respect to the renormalized trace
$\alpha^{-1}\tau_\Mcal\restrict_{p\Mcal p}$.
Let $\beta$ be the trace-preserving automorphism of $\Qc$ obtained from Lemma~\ref{lem:Cpx+B} and satisfying $\beta(p)=p$, $\beta(\Nc)=\Nct$ and $\beta(\Nct)=\Nc$.
Since $A$ and $p\Qc p$ are free, there is a trace-preserving automorphism $\gammat$ of $p\Mcal p$ mapping $A$ identically to itself and whose restriction to $p\Qc p$ agrees with $\beta$.

We claim that $\gammat$ extends to an automorphism of $\Mcal$ whose restriction to $\Qc$ is $\beta$.
Indeed, we may choose a maximal family $(v_i)_{i\in I}$ of partial isometries in $\Qc$ subject to the conditions 
\begin{align*}
\forall i\quad v_i^*v_i\le p\,&\text{ and }\,v_iv_i^*\le 1-p, \\
i,j\in I,\,i\ne j&\implies v_iv_i^*\perp v_jv_j^*.
\end{align*}
Then
\[
\lspan\left((1-p)\Qc(1-p)\cup p\Mcal p\cup\bigcup_{i\in I}(v_ip\Mcal p\cup p\Mcal pv_i^*)\right)
\]
is dense in $\Mcal$.
Indeed, given $x\in \Mcal$, we have
\[
x=pxp+(1-p)x(1-p)+\sum_{i\in I}\big(v_i(v_i^*xp)+(pxv_i)v_i^*\big),
\]
where the sum converges in s.o.t.
We may now define
\[
\gamma(x)=\gammat(pxp)+\beta((1-p)x(1-p))+\sum_{i\in I}\big(\beta(v_i)\gammat(v_i^*xp)+\gammat(pxv_i)\beta(v_i^*)\big).
\]
Using that $\gammat$ and $\beta$ are trace-preserving automorphisms and that $\gammat\restrict_{p\Qc p}=\beta\restrict_{p\Qc p}$,
we can show that $\gamma$ is a trace-preserving automorphism of $\Mcal$.
It satisfies the desired properties by construction.
\end{proof}

\begin{prop}\label{prop:Ai*}
Let $I$ be a finite or countable set and let $\alpha_i>0$ satisfy $\sum_{i\in I}^\infty\alpha_i=1$.
For each $i\in I$, let $B_i$ be a tracial von Neumann algebra.
Let
\[
\Mcal=\left(\bigoplus_{i\in I}\smdp{B_i}{\alpha_i}{p_i}\right)*\left(\smd{\bigfreeprod}{i \in I} (\smdp\Cpx{\alpha_i}{q_i}\oplus\Cpx)\right)
\]
and
\[
\Pc=\left(\bigoplus_{i\in I}\smdp{\Cpx}{\alpha_i}{r_i}\right)*\left(\smd{\bigfreeprod}{i \in I}(\smdp{B_i}{\alpha_i}{s_i}\oplus\Cpx)\right).
\]
Then there is a normal, trace-preserving $*$-isomorphism $\pi:\Mcal\to\Pc$ such that for all $i\in I$, $\pi(p_i)=r_i$.
\end{prop}
\begin{proof}
If $|I|=1$, there is nothing to show, so assume $|I|\ge2$.
Let
\[
\Qc=\left(\bigoplus_{i\in I}\smdp{B_i}{\alpha_i}{r_i}\right)*\left(\smd{\bigfreeprod}{i \in I}(\smdp{\Bt_i}{\alpha_i}{s_i}\oplus\Cpx)\right),
\]
where $\Bt_i$ is a copy of $B_i$.
Let $\pi_0:\Mcal\to\Qc$ be the trace-preserving, injective, normal $*$-homomorphism sending $\bigoplus_{i\in I}B_i$ identically to itself
and sending each projection $q_i$ to $s_i$.
Fixing $j\in I$, we may write
\[
\bigoplus_{i\in I}\smdp{B_i}{\alpha_i}{r_i}=\smdp{A_j}{1-\alpha_j}{1-r_j}\oplus\smdp{B_j}{\alpha_j}{r_j},
\]
where
\[
A_j=\bigoplus_{i\in I\setminus\{j\}}\smdp{B_i}{\alpha_i/(1-\alpha_j)}{r_i}.
\]
Consider the subalgebra
\[
\Qc_j=(\smdp{A_j}{1-\alpha_j}{1-r_j}\oplus\smdp{B_j}{\alpha_j}{r_j})*(\smdp{\Bt_j}{\alpha_j}{s_j}\oplus\smd\Cpx{1-\alpha_j})
\]
of $\Qc$.
By Lemma~\ref{lem:A+B}, there is a trace-preserving automorphism $\gammat_j$ of $\Qc_j$ that maps $A_j\oplus\Cpx$ identically to itself and such that,
letting
\begin{align*}
\Nc_j&=(\Cpx\oplus\smdp{B_j}{\alpha_j}{r_j})*(\smdp\Cpx{\alpha_j}{s_j}\oplus\Cpx)\subseteq\Qc_j \\
\Nct_j&=(\Cpx\oplus\smdp\Cpx{\alpha_j}{r_j})*(\smdp{\Bt_j}{\alpha_j}{s_j}\oplus\Cpx)\subseteq\Qc_j,
\end{align*}
we have
$\gammat_j(\Nc_j)=\Nct_j$ and $\gammat_j(\Nct_j)=\Nc_j$.
We extend $\gammat_j$ to an automorphism $\gamma_j$ of $\Qc$ by taking the free product with the identity maps on all the $\Bt_i\oplus\Cpx$ with $i\ne j$.
Note that, if $i\ne j$, then $\gamma_j$ sends each of 
\[
\smdp\Cpx{1-\alpha_i}{1-r_i}\oplus\smdp{B_i}{\alpha_i}{r_i}\quad\text{and}\quad\smdp{\Bt_i}{\alpha_i}{s_i}\oplus\smdp\Cpx{1-\alpha_i}{1-s_i}
\]
identically to itself, so sends each of $\Nc_i$ and $\Nct_i$ identically to itself, while $\gamma_j(r_j)=r_j$ and, of course, $\gamma_j(\Nc_j)=\Nct_j$ and $\gamma_j(\Nct_j)=\Nc_j$.

Let
\[
\Mcal_i=(\smdp{B_i}{\alpha_i}{p_i}\oplus\smdp\Cpx{1-\alpha_i}{1-p_i})*(\smdp\Cpx{\alpha_i}{q_i}\oplus\smdp\Cpx{1-\alpha_i}{1-q_i})\subseteq\Mcal.
\]
Then
\[
\Mcal=W^*(\bigcup_{i\in I}\Mcal_i).
\]
Moreover, $\pi_0(\Mcal_i)=\Nc_i$.
Suppose that $I$ is finite, say, $I=\{1,2,\ldots,n\}$.
Let
\[
\pi=\gamma_n\circ\gamma_{n-1}\circ\cdots\circ\gamma_2\circ\gamma_1\circ\pi_0:\Mcal\to\Qc.
\]
Note that $\pi$ is a trace-preserving, injective, normal $*$-homomorphism.
Furthermore, for each $j\in I$,
\begin{equation}\label{eq:piMj}
\begin{aligned}
\pi(\Mcal_j)&=\gamma_n\circ\gamma_{n-1}\circ\cdots\circ\gamma_2\circ\gamma_1(\Nc_j) \\
&=\gamma_n\circ\gamma_{n-1}\circ\cdots\circ\gamma_j(\Nc_j) \\
&=\gamma_n\circ\gamma_{n-1}\circ\cdots\circ\gamma_{j+1}(\Nct_j)=\Nct_j.
\end{aligned}
\end{equation}
Thus,
\[
\pi(\Mcal)=W^*(\bigcup_{j=1}^n\pi(\Mcal_j))=W^*(\bigcup_{j=1}^n\Nct_j)=\Pc,
\]
and we are done in the case that $I$ is finite.

Now suppose that $I$ is infinite, say, $I=\{1,2\ldots\}$.
For each $n\ge1$, let
\[
\pi_n=\gamma_n\circ\gamma_{n-1}\circ\cdots\circ\gamma_2\circ\gamma_1\circ\pi_0:\Mcal\to\Qc.
\]
Of course, each $\pi_n$ is an injective, trace-preserving, normal $*$-homomorphism.
Modifying the argument seen at equation~\eqref{eq:piMj}, we get
\[
\pi(\Mcal_j)=\begin{cases}\Nc_j,&\text{if }n<j\\[1ex] \Nct_j,&\text{if }n\ge j\end{cases}
\]
and, furthermore, since the restriction of $\gamma_n$ to $\Nct_j$ is the identity map when $n>j$,
we see that for all $x\in\bigcup_{j=1}^\infty\Mcal_j$, the sequence $(\pi_n(x))_{n=1}^\infty$ is eventually constant.
Since $\bigcup_{j=1}^\infty\Mcal_j$ is dense in $\Mcal$, we may define an injective, trace perservion, normal $*$-homomorphism $\pi:\Mcal\to\Qc$
by
\[
\pi(x)=\lim_{n\to\infty}\pi_n(x),
\]
for all $x\in\Mcal$, where the limit is in s.o.t.
Now we have
\[
\pi(\Mcal)=W^*(\bigcup_{j=1}^\infty\pi(\Mcal_j))=W^*(\bigcup_{j=1}^\infty\Nct_j)=\Pc.
\]
\end{proof}

We are finally ready to prove the generalization of Theorem~\ref{thm:A_infty_free_prod}.

\begin{thm}\label{thm:finitessum}
Fix a self-symmetric tracial von Neumann algebra $A$. Assume for each $i \in \mathbb N$, $s_i \in (0, \infty)$, $r_i \in (1-s_i, \infty]$, and let
$s=\sum_{i \in \mathbb N} s_i$.
Then 
\[
\underset{i \in \mathbb N}{\bigfreeprod} \mathcal F_{s_i,r_i} \cong \Fc_{s,\infty}.
\]
\end{thm}
\begin{proof}
If $s=\infty$, then this was proved in Theorem~\ref{thm:A_infty_free_prod}, so let us assume $s<\infty$.
Arguing as at the start of the proof of Theorem~\ref{thm:A_infty_free_prod}, we may without loss of generality assume $r_i=\infty$ for all $i$.
Now, employing the rescaling formula~\eqref{eq:F_rescale}, proved in Theorem~\ref{thm:F_rescale}, and Theorem~1.5 of~\cite{DR00} (directly for compression by $s$ if $0 < s < 1$, or indirectly for compression by $1/s$ if $s > 1$), we have
\[
\left( \smd \bigfreeprod {i \in \Nats} \Fc_{s_i,\infty}\right)^s \cong \smd \bigfreeprod {i \in \Nats} \Fc_{\frac{s_i}s,\infty},
\]
while $
(\Fc_{s,\infty})^s \cong \Fc_{1,\infty}$.
Hence, we may without loss of generality assume $s=1$.
Using Proposition~\ref{prop:misc_results},
\[
\Fc_{s_i,\infty}=(\smd A{s_i}\oplus\Cpx)*L(\mathbb F_\infty).
\]
Employing Proposition~\ref{prop:Ai*}, we get
\begin{equation}\label{eq:fpsum}\begin{aligned} 
\underset{i \in \mathbb N}{\bigfreeprod} \mathcal F_{s_i,r_i}
&\cong \smd \bigfreeprod {i \in \Nats}\left((\smd A{s_i}\oplus\Cpx)*L(\mathbb F_\infty)\right) \\
&\cong\left(\smd \bigfreeprod {i \in \Nats}(\smd A{s_i}\oplus\Cpx)\right)*\left(\bigoplus_{i\in\Nats}\smd\Cpx{s_i}\right)*L(\mathbb F_\infty) \\
&\cong\left(\bigoplus_{i\in\Nats}\smd A{s_i}\right)*\left(\smd \bigfreeprod {i \in \Nats}(\smd \Cpx{s_i}\oplus\Cpx)\right)*L(\mathbb F_\infty) \\
&\cong A*L(\mathbb F_\infty) \cong \Fc_{1,\infty}
\end{aligned}
\end{equation}
where in the second last isomorphism above we used the self-symmetry property of $A$ and Lemma \ref{lem:direct_sum_self_symm}.
\end{proof}

We conclude our paper with the following theorem that uses the full power of Proposition \ref{prop:Ai*}.

\begin{thm}
	Let $I$ be a finite or countably infinite index set. Let $\{A_i\}_{i \in I}$ be self-symmetric tracial von Neumann algebras. For each $i \in I$, let $s_i \in (0, \infty)$, $r_i \in (1-s_i, \infty]$. Assume $s = \sum_{i \in I} s_i < \infty$, and denote $r = \sum_{i \in I} r_i$, 
	\[
		A = \bigoplus_{i \in I} \smd{A_i}{s_i / s}.
	\]
	Then
	\[
		\smd{\bigfreeprod}{i \in I} \Fc_{s_i, r_i}(A_i) \cong \mathcal{F}_{s,r}(A).
	\]

\end{thm}

\begin{proof}
	When $I$ is infinite, $r = \infty $ necessarily. Then the result follows from the same proof as of Theorem \ref{thm:finitessum}, except in the series of isomorphisms~\eqref{eq:fpsum}, we replace each
	\[
	\smd A{s_i}\quad\text{with}\quad\smd{A_i}{s_i}.
	\]

	When $I$ is finite, it suffices to show the case $|I|=2$, and then the theorem follows by induction on $|I|$. In this case, $A$ is
	\[
		\underset{t}{A_1} \oplus \underset{1-t}{A_2}
	\]
	where $t = s_1/s$. We realize $\mathcal{F}_{s,r}(A)$ by choosing $n \in \mathbb{N}$ large enough so that $(st/n)^{-2} -1 > 1$, and that
	\[
		(\underset{t}{\mathbb{C}} \oplus \underset{1-t}{\mathbb{C}})^{*n}
	\]
	is a factor, in addition to 
	\[
		\mathcal{F}_{s,r}(A) \cong (A^{*n} * L\mathbb{F}_{\frac{s+r-1}{s^2}n^2-n+1})^{n/s}.
	\]
	We use Theorem 1.5 of \cite{DR00} and the rescaling formula (\ref{eq:F_rescale}):
	\begin{align*}
		(\mathcal{F}_{st, r_1}(A_1) & * \mathcal{F}_{s(1-t), r_2}(A_2))^{st/n} \\
		& \cong (\mathcal{F}_{st, r_1}(A_1))^{st/n} * (\mathcal{F}_{s(1-t), r_2}(A_2))^{st/n} * L\mathbb{F}_{(st/n)^{-2}-1} \\
		& \cong \mathcal{F}_{n,0}(A_1) * (\mathcal{F}_{n,1 + \frac{2nt(1-t)-1}{(1-t)^2}}(A_2))^{\frac{t}{1-t}} * L\mathbb{F}_{t^{-2}(\frac{s+r-1}{s^2}n^2-n+1)}.
	\end{align*}
	Then we use the same technique as the proof of Lemma \ref{lem:F_s+n,r} to compute
	\[
		((\underset{t}{A_1} \oplus \underset{1-t}{A_2})^{*n})^t \cong A_1^{*n} * ((\underset{t}{\mathbb{C}} \oplus \underset{1-t}{A_2})^{*n})^t
	\]
	and
	\[
		((\underset{t}{\mathbb{C}} \oplus \underset{1-t}{A_2})^{*n})^{1-t} \cong A_2^{*n} * ((\underset{t}{\mathbb{C}} \oplus \underset{1-t}{\mathbb{C}})^{*n})^{1-t} \cong \mathcal{F}_{n, 1+\frac{2nt(1-t)-1}{(1-t)^2}}(A_2).
	\]
	The above two lines imply the following:
	\[
		((\underset{t}{A_1} \oplus \underset{1-t}{A_2})^{*n})^t \cong A_1^{*n} * (\mathcal{F}_{n, 1+\frac{2nt(1-t)-1}{(1-t)^2}}(A_2))^{\frac{t}{1-t}}.
	\]
	Finally, using Theorem 1.5 of \cite{DR00}, we compute
	\begin{align*}
		(\mathcal{F}_{s,r}(A))^{st/n} & \cong (A^{*n} * L\mathbb{F}_{\frac{s+r-1}{s^2}n^2-n+1})^{t} \\
		& \cong (A^{*n})^t * (L\mathbb{F}_{\frac{s+r-1}{s^2}n^2-n+1})^t * B_{t^{-2}-1} \\
		& \cong (A^{*n})^t * L\mathbb{F}_{t^{-2}(\frac{s+r-1}{s^2}n^2-n+1)},
	\end{align*}
	where $B_{t^{-2}-1}$ is a tracial von Neumann algebra having free dimension $t^{-2}-1$. Now putting them all together, we obtain
	\[
		(\mathcal{F}_{s,r}(A))^{st/n} \cong (\mathcal{F}_{st, r_1}(A_1) * \mathcal{F}_{s(1-t), r_2}(A_2))^{st/n}.
	\]
	Rescaling will then prove the case $|I|=2$.   
\end{proof}


%


\end{document}